\documentclass{amsart}
\usepackage{amsfonts,amssymb,stmaryrd,amscd,amsmath,latexsym,amsbsy}
\numberwithin{equation}{section}
\theoremstyle{plain}
\newtheorem{thm}{Theorem}[section]

\newtheorem{prop}[thm]{Proposition}

\newtheorem{lem}[thm]{Lemma}
\newtheorem{cor}[thm]{Corollary}
\newtheorem{conjecture}[thm]{Conjecture}

\theoremstyle{remark}
\newtheorem{rema}[thm]{Remark}

\newcommand{\ad}{{\mbox{\upshape{ad}}}}
\newcommand{\Ad}{{\mbox{\upshape{Ad}}}}

\newcommand{\afrak}{{\mathfrak a}}

\newcommand{\Aut}{\mathrm{Aut}}

\newcommand{\bfrak}{{\mathfrak b}}

\newcommand{\C}{{\mathbb C}}

\newcommand{\N}{{\mathbb N}}

\newcommand{\cM}{{\mathcal M}}

\newcommand{\dfrak}{{\mathfrak d}}

\newcommand{\efrak}{{\mathfrak e}}

\newcommand{\ffrak}{{\mathfrak f}}

\newcommand{\gfrak}{{\mathfrak g}}

\newcommand{\hfrak}{{\mathfrak h}}

\newcommand{\id}{{\mbox{id}}}

\newcommand{\kfrak}{{\mathfrak k}}
\newcommand{\kow}{{\varDelta}}

\newcommand{\Mat}{\mathrm{Mat}}

\newcommand{\ofrak}{{\mathfrak o}}

\newcommand{\os}{{\overline{s}}}
\newcommand{\ot}{\otimes}

\newcommand{\qfield}{k}
\newcommand{\pfrak}{{\mathfrak p}}

\newcommand{\slfrak}{{\mathfrak{sl}}}
\newcommand{\sofrak}{{\mathfrak{so}}}
\newcommand{\spfrak}{{\mathfrak{sp}}}

\newcommand{\uqg}{{U_q(\mathfrak{g})}}

\newcommand{\uqpk}{{U_q'(\mathfrak{k})}}
\newcommand{\uqpks}{{U_q'(\mathfrak{k})_s}}
\newcommand{\uqtwk}{{U_q^{\mathrm{tw}}(\mathfrak{k})}}

\newcommand{\vep}{\varepsilon}

\newcommand{\Z}{{\mathbb Z}}

\begin{document}

\title[Braid group actions on coideal subalgebras]{Braid group actions on coideal subalgebras of quantized enveloping algebras}

\author{ Stefan Kolb}
\address{Stefan Kolb, School of Mathematics and Statistics, Newcastle University, Newcastle upon Tyne NE1 7RU, UK }
\email{stefan.kolb@ncl.ac.uk}

\author{ Jacopo Pellegrini}
\address{Jacopo Pellegrini, Dipartimento di Matematica Pura ed Applicata, Universit\`a degli Studi di Padova, Via Trieste 63, 35121 Padova, Italy}
\email{pellej14@hotmail.com }

\thanks{Both authors are very grateful to J.V.~Stokman for numerous mathematical discussions and support. They also thank I.~Heckenberger for verifying the G2 braid relation in Section \ref{sec:Chevalley} with the help of the computer algebra program FELIX. This research project started out while J.~Pellegrini was an LLP/Erasmus visiting student at the University of Amsterdam in the first half of 2010. He is grateful to the Korteweg-de Vries Institute for hospitality. The work of S.~Kolb was supported by the Netherlands Organization for Scientific Research (NWO) within the VIDI-project "Symmetry and modularity in exactly solvable models".}

\subjclass[2000]{17B37}
\keywords{Quantized enveloping algebras, braid groups, coideal subalgebras, quantum symmetric pairs}

\begin{abstract}
  We construct braid group actions on coideal subalgebras of quantized enveloping algebras which appear in the theory of quantum symmetric pairs. In particular, we construct an action of the semidirect product of $\Z^n$ and the classical braid group in $n$ strands on the coideal subalgebra corresponding to the symmetric pair $(\slfrak_{2n}(\C),\mathfrak{sp}_{2n}(\C))$. This proves a conjecture by Molev and Ragoucy. We expect similar actions to exist for all symmetric Lie algebras. 
The given actions are inspired by Lusztig's braid group action on quantized enveloping algebras and are defined explicitly on generators. Braid group and algebra relations are verified with the help of the package \texttt{Quagroup} within the computer algebra program \texttt{GAP}. 
\end{abstract}

\maketitle
\section{Introduction}
  In the theory of quantum groups an important role is played by Lusztig's braid group action on the quantized enveloping algebra $\uqg$ of a complex simple Lie algebra $\gfrak$ \cite{a-Lusztig90}, \cite{b-Lusztig93}. This braid group action allows the definition of root vectors and Poincar\'e-Birkhoff-Witt bases. It is ubiquitous in the representation theory of $\uqg$ and appeared for instance in the investigation of canonical bases \cite{a-Lusztig96}, the construction of quantum Schubert cells \cite{a-DCoKacPro95}, and the classification of coideal subalgebras \cite{a-HeckSchn09p}.
  
Let $\theta:\gfrak\rightarrow \gfrak$ be an involutive Lie algebra automorphism, that is $\theta^2=\id$, and let $\kfrak$ be the Lie subalgebra of $\gfrak$ consisting of elements fixed under $\theta$. In a series of papers G.~Letzter constructed and investigated quantum group analogs of $U(\kfrak)$ as one-sided coideal subalgebras $\uqpk$ of $\uqg$ \cite{a-Letzter99a}, \cite{MSRI-Letzter}. The algebras $\uqpk$ can be given explicitly in terms of generators and relations \cite{a-Letzter03} and encompass various classes of quantum analogs of $U(\kfrak)$ which had been constructed previously. We call the algebras $\uqpk$ quantum symmetric pair coideal subalgebras. 
For $\gfrak$ of classical type a different construction was given by Noumi and his collaborators \cite{a-Noumi96}, \cite{a-NS95}, \cite{a-Dijk96}. In the influential paper \cite{a-Noumi96} Noumi constructed quantum algebras $U'_q(\sofrak_n)$ and $U'_q(\mathfrak{sp}_{2n})$ corresponding to the symmetric pairs $(\slfrak_n(\C),\sofrak_n(\C))$ and $(\slfrak_{2n}(\C),\mathfrak{sp}_{2n}(\C))$, respectively. Even earlier the algebra $U'_q(\sofrak_n)$ had appeared in the work of Gavrilik and Klimyk \cite{a-GavKlimyk91}. The coideal subalgebras $U'_q(\sofrak_n)$ and $U'_q(\mathfrak{sp}_{2n})$ are special examples of quantum symmetric pair coideal subalgebras.

Recently, examples of braid group actions on quantum symmetric pair coideal subalgebras $\uqpk$ appeared in the literature. Let $Br(\afrak_{n-1})$ denote the classical braid group in $n$ strands, that is, the braid group corresponding to Dynkin type $A_{n-1}$. Molev and Ragoucy \cite{a-MolRag08}, and independently Chekhov \cite{a-Chekhov07}, constructed an action of $Br(\afrak_{n-1})$ on $U'_q(\sofrak_n)$ by algebra automorphisms. This action is a quantum analog of the action of the symmetric group $S_n$ on $\sofrak_n(\C)$ by simultaneous permutation of rows and columns. By similar reasoning Molev and Ragoucy conjectured that the action of the semidirect product $(\Z/4\Z)^n \rtimes S_n$ on $\mathfrak{sp}_{2n}(\C)$ has a quantum analog. In the present paper we verify this conjecture.
\begin{thm}\label{conj:MR} {\upshape(\cite[Conjecture 4.7]{a-MolRag08})}
  There exists an action of the group $\Z^n\rtimes Br(\afrak_{n-1})$ on $U_q'(\mathfrak{sp}_{2n})$ by algebra automorphisms which is a quantum analog of the action of $(\Z /4\Z)^n \rtimes S_n$ on $\mathfrak{sp}_{2n}(\C)$.
\end{thm}
The aim of this paper is to understand the action of $Br(\afrak_{n-1})$ on $U_q'(\sofrak_n)$ and the action of  $\Z^n\rtimes Br(\afrak_{n-1})$ on $U_q'(\mathfrak{sp}_{2n})$ within the general theory of quantum symmetric pairs.
More specifically, let $\{\alpha_i\,|\, i\in I\}$ be a set of simple roots for the root system of $\gfrak$ and let $\{s_i\,|\,i\in I\}$ denote the generators of the corresponding braid group $Br(\gfrak)$. Recall that involutive automorphisms of $\gfrak$ are classified in terms of pairs $(X,\tau)$ where $X\subset I$ and $\tau$ is a diagram automorphism of $\gfrak$ of order at most two \cite{a-Araki62}. We write $Br_X$ to denote the subgroup of $Br(\gfrak)$ generated by all $s_i$ with $i\in X$. Let, moreover, $\Sigma$ denote the restricted root system corresponding to $\theta$ and let $Br(\Sigma,\theta)$ denote the corresponding braid group. One can show that there exists a natural action of a semidirect product $Br_X \rtimes Br(\Sigma,\theta)$ on $\kfrak$. Extending Molev's and Ragoucy's original conjecture, we expect that this action has a quantum analog.
\begin{conjecture}\label{conj:KP} 
  There exists an action of the group $Br_X\rtimes Br(\Sigma,\theta)$ on $\uqpk$ by algebra automorphisms which is a quantum analog of the action of $Br_X\rtimes Br(\Sigma,\theta)$ on $\kfrak$.
\end{conjecture}
In the present paper we prove Conjecture \ref{conj:KP} for the following three example classes.
  \begin{align*}
    &\mbox{(I)} & &\gfrak\mbox{ arbitrary, }  X=\emptyset, \mbox{ and }  \tau=\id,\\
    &\mbox{(II)} & &\gfrak\mbox{ arbitrary, } X=\emptyset,  \mbox{ and }\tau\neq\id,\\
    &\mbox{(III)}& &\gfrak=\slfrak_{2n}(\C),\, X=\{1,3,5,\dots,2n-1\}, \mbox{ and } \tau=\id,
  \end{align*}
where in case (III) we use the standard ordering of simple roots.  
In case (I) the involution $\theta$ coincides with the Chevalley automorphism $\omega$ of $\gfrak$, and in case (II) one has $\theta=\tau\circ\omega$ where $\tau$ is a nontrivial diagram automorphism.   
Proving Conjecture \ref{conj:KP} in case (III) also proves Theorem \ref{conj:MR}.  Indeed, in this case $Br_X=\Z^n$, the restricted root system $\Sigma$ is of type $A_{n-1}$, and the quantum symmetric pair coideal subalgebra $\uqpk$ coincides with Noumi's $U_q'(\mathfrak{sp}_{2n})$, see Remark \ref{rem:NoumiMolevLetzter}.    

The construction of the action of $Br(\Sigma,\theta)$ on $\uqpk$ is guided by Lusztig's braid group action on $\uqg$. There exists a natural group homomorphism
\begin{align*}
  i_{\Sigma,\theta} : Br(\Sigma,\theta) \rightarrow Br(\gfrak)
\end{align*}
but $\uqpk$ is not invariant under the Lusztig action of $ i_{\Sigma,\theta}(Br(\Sigma,\theta))$. Nevertheless, in the example classes (I), (II), and (III) above, it is possible to modify the restriction of the Lusztig action of $i_{\Sigma,\theta}(Br(\Sigma,\theta))$ to $\uqpk$ in such a way that $\uqpk$ is mapped to itself. To this end, following \cite{a-Letzter03}, we write the algebra $\uqpk$ explicitly in terms of generators and relations. We then make an ansatz for the action of the generators of $Br(\Sigma,\theta)$ on the generators of $\uqpk$. The fact that this ansatz actually defines algebra automorphisms of $\uqpk$ which satisfy the braid relations is verified by computer calculations using de Graaf's package \texttt{Quagroup} \cite{a-quagroup} within the computer algebra system \texttt{GAP} \cite{GAP4}. The need for computer calculations should not be too surprising. In Lusztig's original work \cite{a-Lusztig90}, \cite{b-Lusztig93} the verification of the braid group action also involved long calculations, and quantum symmetric pair coideal subalgebras $\uqpk$ tend to feature more involved relations than the quantized enveloping algebra $\uqg$. As the list of symmetric pairs in \cite{a-Araki62} is finite, one could well try to establish Conjecture \ref{conj:KP} for general $\theta$ by a case by case analysis. Given the computational complexity of the examples considered in this paper, however, a general proof would be more desirable. Our results give strong evidence that the statement of Conjecture \ref{conj:KP} holds.

In Chekhov's work \cite{a-Chekhov07} the algebra $U_q'(\sofrak_{n+1})$ is called the quantum $A_n$-algebra. It appears as a deformed algebra of geodesic functions on the Teichm{\"u}ller space of a disk with $n$ marked points on the boundary. In this setting the braid group action comes from the action of the mapping class group. The paper \cite{a-Chekhov07} also contains a second quantum algebra, called the quantum $D_n$-algebra. It would be interesting to know if the quantum $D_n$-algebra also coincides with a quantum symmetric pair coideal subalgebra and whether the natural braid group action Chekhov obtains from quantum Teichm{\"u}ller theory can be related to Lusztig's braid group action. 

The present paper focuses on the example classes (I), (II), and (III) and does not attempt maximal generality. In Section \ref{sec:prel} we recall the braid group action of $B(\Sigma,\theta)$ on $\kfrak$ and fix notation for quantum groups and quantum symmetric pairs. In Sections \ref{sec:Chevalley}, \ref{sec:invol}, \ref{sec:wXom} we prove Conjecture \ref{conj:KP} for the example classes (I), (II), (III), respectively. The \texttt{GAP}-codes used to check all relations are available for download from \cite{WWW}. Each of Sections \ref{sec:Chevalley} and \ref{sec:invol} ends with a short overview over the respective \texttt{GAP}-code. Section \ref{sec:Chevalley} contains the results of the second named author's master thesis \cite{master-Pellegrini10} which was written under the guidance of J.V.~Stokman and the first named author. 

\section{Preliminaries}\label{sec:prel}
\subsection{Braid group action for symmetric pairs}
Let $\gfrak$ be a complex simple Lie algebra with Cartan subalgebra
$\hfrak$. Let $\Phi\subset \hfrak^\ast$ denote the corresponding root
system and fix a set $\Pi=\{\alpha_i\,|\,i\in I\}$ of simple
roots. Write $W$ to denote the Weyl group generated by all reflections
$s_{\alpha_i}$ for $i\in I$. Let $(\cdot,\cdot)$ denote the
$W$-invariant scalar product on the real vector space spanned
by $\Phi$ such that all short roots $\alpha$ satisfy
$(\alpha,\alpha)=2$. As usual, let
$a_{ij}=2(\alpha_i,\alpha_j)/(\alpha_i,\alpha_i)$ denote the entries 
of the Cartan matrix of $\gfrak$.  For $i,j\in I$ let $m_{ij}$ denote
the order of $s_{\alpha_i} s_{\alpha_j}$ in $W$. Let $Br(\gfrak)$
denote the Artin braid group corresponding to $W$. More  
explicitly, $Br(\gfrak)$ is generated by elements $\{s_i\,|\,i\in I\}$
and relations 
\begin{align}\label{eq:braid-rels}
  \underbrace{s_i s_j s_i s_j \cdots}_{m_{ij}\mathrm{ factors}} =
  \underbrace{s_j s_i s_j s_i \cdots.}_{m_{ij} \mathrm{ factors}} 
\end{align}
The braid group $Br(\gfrak)$ acts on $\gfrak$ by  Lie algebra
automorphisms. Let $\{e_i,f_i,h_i\,|\,i\in I\}$ be a set of Chevalley
generators for $\gfrak$. For $i\in I$ define 
\begin{align}\label{eq:Adsi}
  \Ad(s_i)=\exp(\ad(e_i)) \exp(\ad(-f_i)) \exp(\ad(e_i))
\end{align}
where the symbol $\ad$ denotes the adjoint action and where $\exp$ is
the exponential series which is well defined on nilpotent elements. By \cite[Lemma 56]{b-Steinberg68} there exists a group homomorphisms  
\begin{align}
  \Ad:Br(\gfrak)\rightarrow \Aut(\gfrak)
\end{align}
such that $\Ad(s_i)$ is given by \eqref{eq:Adsi}. 

Let now $\theta:\gfrak\rightarrow \gfrak$ be an involutive Lie algebra
automorphism and let $\gfrak=\kfrak\oplus \pfrak$ be the corresponding
decomposition into the $+1$ and the -1 eigenspace of $\theta$, that is $\kfrak=\{x\in
\gfrak\,|\,\theta(x)=x\}$. In this paper we consider the
following three classes of examples. 
\begin{enumerate}
  \item[{(I)}] Let $\gfrak$ be arbitrary and let $\theta$ be the
    Chevalley automorphism $\omega\in \Aut(\gfrak)$ defined by 
  \begin{align}\label{eq:om-def}
     \omega(e_i)&=-f_i, & \omega(f_i)&=-e_i, & \omega|_{\hfrak}&=-\id_\hfrak.
  \end{align} 
  In this case, if $\gfrak=\slfrak_n(\C)$ then $\kfrak\cong \sofrak_n(\C)$.
  \item[{(II)}] Let $\gfrak$ be arbitrary and let
    $\theta=\tau\circ\omega$ for a nontrivial diagram automorphism
    $\tau$ of order 2. A nontrivial diagram automorphism only exists
    if $\gfrak$ is of type $A_n$ for $n\ge 2$, of type $D_n$ for $n\ge
    4$, or of type $E_6$. 
  \item[{(III)}] Let $m\in \N$ and $\gfrak=\slfrak_{2m}(\C)$ with the
    standard ordering of the simple roots. We consider
    $\theta=\Ad(w_X)\circ\omega$ where $w_X=s_1 s_3 s_5 \cdots
    s_{2m-1}$. In this case $\kfrak\cong \spfrak_{2m}(\C)$. 
\end{enumerate}
Any diagram automorphism $\tau$ for $\gfrak$ yields a group
automorphism of $Br(\gfrak)$ which we denote by the same symbol
$\tau$. On the generators of $Br(\gfrak)$ one has 
$\tau(s_i)=s_{\tau(i)}$. Now define
\begin{align*}
  Br(\gfrak,\theta)=\begin{cases}
                      Br(\gfrak)&\mbox{if
                        $(\gfrak,\theta)=(\gfrak,\omega)$\quad
                        (I),}\\ 
                      \{b\in Br(\gfrak)\,|\,\tau(b)=b\}& \mbox{if
                        $(\gfrak,\theta)=(\gfrak,\tau\circ\omega)$\quad
                        (II),}\\ 
          \{b\in Br(\gfrak)\,|\,w_X b=b w_X\}& \mbox{if
            $(\gfrak,\theta)=(\slfrak_{2m}(\C),\Ad(w_X){\circ}\omega)$\,
            (III).}  
                    \end{cases}
\end{align*}
\begin{lem}
  Under the action $\Ad$, the subgroup $Br(\gfrak,\theta)$ of
  $Br(\gfrak)$ maps $\kfrak$ to itself.
\end{lem}
\begin{proof}
  One first observes that $\omega$ commutes with $\Ad(s_i)$ for all
  $i\in I$. In case (II) with $\theta=\tau\circ\omega$ one verifies
  that $\theta(\Ad(b)(x))=\Ad(\tau(b))(\theta(x))$ holds for all $b\in
  Br(\gfrak)$, $x\in \gfrak$. Similarly, in case (III) one
  has $\theta(\Ad(b)(x))=\Ad(w_X b w_X^{-1})(\theta(x))$. The above
  observations imply that for all $b\in Br(\gfrak,\theta)$ one has
  $\Ad(b)\circ \theta=\theta\circ \Ad(b)$. Hence for $b\in
  Br(\gfrak,\theta)$ one has $\Ad(b)(\kfrak)=\kfrak$.  
\end{proof}
In the following we use the standard ordering of simple roots as in
\cite{b-Bourbaki4-6}. In case (II), that is for $\theta=\tau\circ\omega$, we need to
distinguish three different cases.
\begin{enumerate}
  \item[{(IIA)}] $\gfrak=\afrak_n=\slfrak_{n+1}(\C)$ and $\tau(i)=n-i+1$.
  \item[{(IID)}] $\gfrak=\dfrak_{n+1}=\sofrak_{2n+2}(\C)$ and 
     \begin{align*} 
       \tau(i)=\begin{cases}
                  i & i\neq n, n+1,\\
                  n & i=n+1,\\
                  n+1 & i=n. 
               \end{cases}
     \end{align*}
  \item[{(IIE)}] $\gfrak=\mathfrak{e}_6$ and $\tau$ is the
    nontrivial diagram automorphism
     \begin{align*} 
       \tau(1)=6,\, \tau(2)=2,\,\tau(3)=5,\,\tau(4)=4,\,\tau(5)=3,\,\tau(6)=1. 
     \end{align*}
\end{enumerate}
Now define for each of the above cases a braid group
$Br(\Sigma,\theta)$ as follows
\begin{align*}
  &\mbox{in case (I):}& Br(\Sigma,\omega)&=Br(\gfrak),\\
  &\mbox{in case (IIA) with $n=2r$:}& Br(\Sigma,\tau\circ\omega)
              &= Br(\bfrak_r),\\
  &\mbox{in case (IIA) with $n=2r{-}
1$:}& Br(\Sigma,\tau\circ\omega)
              &= Br(\bfrak_r),\\
  &\mbox{in case (IID):}& Br(\Sigma,\tau\circ\omega)
              &= Br(\bfrak_{n}),\\
  &\mbox{in case (IIE):}& Br(\Sigma,\tau\circ\omega)&= Br(\ffrak_4),\\
  &\mbox{in case (III):} & Br(\Sigma,\Ad(w_X)\circ \omega)&=Br(\afrak_{m-1}).
\end{align*}
In the following we will consider the braid groups $Br(\gfrak)$ and
$Br(\Sigma,\theta)$ simultaneously. To avoid confusion we denote the
generators of $Br(\Sigma, \theta)$ by $\os_i$ as opposed to the
notation $s_i$ for the generators of $Br(\gfrak)$.
\begin{prop}\label{prop:BrSinBrTheta}
  There exists a group homomorphism
  \begin{align*}
    i_{\Sigma,\theta}:Br(\Sigma,\theta)\rightarrow Br(\gfrak, \theta)
  \end{align*}
  determined in each of the cases (I), (II), and (III) as follows:
  \begin{enumerate}
     \item[{(I)\phantom{IA}}] $Br(\gfrak)\rightarrow
       Br(\gfrak),$\qquad $\os_i\mapsto 
       s_i$,
     \item[{(IIA)}]If $n=2r$:
       \begin{align*}
         Br(\bfrak_r)&\rightarrow Br(\afrak_{2r}), & 
                  \os_i\mapsto \begin{cases}
                                 s_is_{n-i+1}& i\neq r,\\
                                 s_r s_{r+1} s_r & i=r.
                               \end{cases}
       \end{align*}
       If $n=2r{-}1$:
       \begin{align*}
         Br(\bfrak_r)&\rightarrow Br(\afrak_{2r-1}), & 
                  \os_i\mapsto \begin{cases}
                                 s_is_{n-i+1}& i\neq r,\\
                                 s_r & i=r.
                               \end{cases}
       \end{align*}  
     \item[{(IID)}]  $Br(\bfrak_n)\rightarrow Br(\dfrak_{n+1})$,\qquad  
                  $\os_i\mapsto \begin{cases}
                                 s_i& i\neq n,\\
                                 s_n s_{n+1} & i=n.
                               \end{cases} $
     \item[{(IIE)}]  $Br(\ffrak_4)\rightarrow Br(\efrak_6)$,\qquad 
       $\os_1\mapsto s_1s_6$,\, $\os_2\mapsto s_3s_5$,\, $\os_3\mapsto s_4$,\,
       $\os_4\mapsto s_2$.
     \item[{(III)\phantom{I}}]  $Br(\afrak_{m-1})\rightarrow
       Br(\afrak_{2m-1})$,\qquad $\os_i\mapsto s_{2i} s_{2i-1}s_{2i+1}
       s_{2i}$.
  \end{enumerate}
\end{prop}
\begin{proof}
  The images of the generators $\os_i$ under $i_{\Sigma,\theta}$ do indeed lie in
  $Br(\gfrak,\theta)$. It is verified by direct computation that the elements $i_{\Sigma,\theta}(\os_i)$ satisfy the braid relations of $Br(\Sigma,\theta)$.
\end{proof}
\begin{cor}\label{cor:Bract}
  In any of the cases (I), (II), and (III) there exists an action of
  $Br(\Sigma,\theta)$ on $\kfrak$ by Lie algebra automorphisms. This
  action is given by the composition of the map $i_{\Sigma,\theta}$ from Proposition 
  \ref{prop:BrSinBrTheta} with the action of $Br(\gfrak)$ on
  $\gfrak$.  
\end{cor}
\begin{rema}\label{rem:restricted}
  Let $\Sigma$ be the restricted root system corresponding to the
  symmetric Lie algebra $(\gfrak,\theta)$, see \cite[2.4]{a-Araki62}. The Dynkin diagram of $\Sigma$ is given by the third column of the table in \cite[p.~32/33]{a-Araki62}, however, $\Sigma$ may be non-reduced. The braid group
  $Br(\Sigma,\theta)$ defined above for special examples is exactly
  the braid group corresponding to the root system $\Sigma$. An
  action of $Br(\Sigma, \theta)$ on $\kfrak$, generalizing the action
  of the above corollary, exists for any symmetric Lie algebra
  $(\gfrak,\theta)$.  
\end{rema}
\begin{rema}
  To compare the classical and the quantum situation for case (III) in Section \ref{sec:sp-bract} we make the action of $Br(\afrak_{m-1})$ on $\mathfrak{sp}_{2m}(\C)$ more explicit. Let $e_i,f_i,h_i$ for $i=1,\dots,2m-1$ denote the standard Chevalley generators of $\slfrak_{2m}(\C)$. Define a $(2m\times 2m)$-matrix S by
\begin{align*}
  S=\left(
        \begin{matrix} J&0&\dots&0\\ 0&J&\dots&0\\ 
                       \vdots &\vdots&\ddots&\vdots\\
                       0&0&\cdots&J  
        \end{matrix}
    \right)\qquad\mbox{where}\quad J=\left(\begin{matrix}
        0&1\\-1&0\end{matrix}\right). 
\end{align*}
For any $x\in \slfrak_{2m}(\C)$ one has $\theta(x)=- \Ad(S)(x^t)$. The Chevalley generators $e_i,f_i, h_i$ for odd $i$ are invariant under $\theta$. Define elements $b_{2j}=f_{2j}+\theta(f_{2j})$ for $j=1,2,3,\dots m-1$. Using the weight decomposition of $\gfrak=\slfrak_{2m}$ one shows that the Lie algebra
$\kfrak\cong\mathfrak{sp}_{2m}(\C)$ is generated by the elements
\begin{align}
  e_i,&f_i,h_i & &\mbox{for $i=1,3,5,\dots,2m-1$,}\\
  &b_{2j}       & &\mbox{for $j=1,2,3,\dots,m-1$.}
\end{align} 
For any ring $R$ and $s\in \N$ let $\Mat_s(R)$ denote the set of $(s\times s)$-matrices with entries in $R$. In view of the special form of $S$ it is natural to consider elements in $\kfrak\cong \mathfrak{sp}_{2m}(\C)$ as elements in $\Mat_m(\Mat_2(\C))$. The action $Br(\afrak_{m-1})$ on $\kfrak$ then factors through the natural action of the symmetric group $S_m$ on $\Mat_m(\Mat_2(\C))$ by simultaneous permutations of rows and columns.
We define $b_j=f_j$ for $j=1,3,5,\dots,2m-1$ and calculate
\begin{align}\label{eq:Adbj}
  \Ad(s_{2i}s_{2i-1}s_{2i+1}s_{2i})(b_j)=\begin{cases} 
        [[b_{2i-2},b_{2i-1}],b_{2i}] & \mbox{if $j=2i-2$,}\\
        b_{2i+1}                     & \mbox{if $j=2i-1$,}\\
        b_j                          & \mbox{if $j=2i$ or $|j-2i|>2$,}\\
        b_{2i-1}                     & \mbox{if $j=2i+1$,}\\
        [[b_{2i+2},b_{2i+1}],b_{2i}] & \mbox{if $j=2i+2$.}
   \end{cases}
\end{align}
In Section \ref{sec:sp-bract} we will construct a quantum group analog of the above action. 
Now consider odd $j=1,3,5,\dots,2m-1$ and observe that the subspace $\kfrak$ is invariant under the action of $\Ad(s_j)$. One has $\Ad(s_j^2)(f_{j+1})=-f_{j+1}$ and $\Ad(s_j)$ commutes with $\theta$ for odd $j$. Hence $\Ad(s_j^2)(b_{j+1})=-b_{j+1}$ 
and the action of $s_j$ on $\mathfrak{sp}_{2m}(\C)$ has order four. In other words, the operators $\Ad(s_j)$ for $j=1,3,5,\dots,2m-1$ give an action of $(\Z/4\Z)^m$ on $\mathfrak{sp}_{2m}(\C)$ by Lie algebra automorphisms.
Taking into account the action of $S_m$ discussed above, one obtains the desired action of $(\Z/4\Z)^m\rtimes S_m$ on $\mathfrak{sp}_{2m}(\C)$. 
\end{rema}
\subsection{Quantum groups}
Let $\qfield$ be a field and let $q\in \qfield\setminus\{0\}$ be not a
root of unity. For technical reasons which will become apparent in Sections \ref{sec:invol} and \ref{sec:wXom} we assume that $\qfield$ contains a square root $q^{1/2}$ of $q$. We consider the quantized enveloping algebra $\uqg$ as the
$\qfield$-algebra with generators $E_i, F_i, K_i, K_i^{-1}$ for
all $i\in I$ and relations given in
\cite[4.3]{b-Jantzen96}. Recall that $\uqg$ is a Hopf algebra
with coproduct $\kow$ determined by 
\begin{align}
  \kow(K_i)&= K_i\ot K_i,\nonumber\\
  \kow(E_i)&= E_i \ot 1 + K_i \ot E_i,\label{eq:coproduct}\\
  \kow(F_i)&= F_i \ot K_i^{-1} + 1 \ot F_i\nonumber
\end{align}
for all $i\in I$. For any $i\in I$ one defines
$q_i=q^{(\alpha_i,\alpha_i)/2}$ and for $n\in \N$ the $q$-number
\begin{align}
  [n]_i=\frac{q_i^n-q_i^{-n}}{q_i-q_i^{-1}}
\end{align}
and the $q$-factorial $[n]_i!=[n]_i [n{-}1]_i \dots [2]_i$. If
$(\alpha_i,\alpha_i)=2$ then we will also write $[n]$ and $[n]!$
instead of $[n]_i$ and $[n]_i!$, respectively.
As observed by Lusztig \cite{a-Lusztig90},  the action of $Br(\gfrak)$
on $\gfrak$ by Lie algebra automorphisms deforms to an action of
$Br(\gfrak)$ on $\uqg$ by algebra automorphisms. The image of the
generator $s_i\in Br(\gfrak)$ under this action is the
Lusztig automorphism $T_i$ as given in \cite[8.14]{b-Jantzen96}.
In the following it is sometimes more convenient to work with the
inverse of $T_i$ which we therefore recall explicitly. One has
\begin{align}
  T_i^{-1}(E_i)&=-K_i^{-1}F_i , \qquad T_i^{-1}(F_i)=-E_i K_i, \qquad
  T_i^{-1}(K_i)=K_i^{-1},\label{eq:L1} 
\end{align}
and
\begin{align}
  T_i^{-1}(K_j)&=K_jK_i^{-a_{ij}},\nonumber\\
  T_i^{-1}(E_j)&=\sum_{s=0}^{-a_{ij}}(-1)^s q_i^{-s}E_i^{(s)}E_j
  E_i^{(-a_{ij}-s)},\label{eq:L2}\\
  T_i^{-1}(F_j)&=\sum_{s=0}^{-a_{ij}}(-1)^s
  q_i^{s}F_i^{(-a_{ij}-s)}F_j F_i^{(s)}.\nonumber 
\end{align}
for all $j\neq i$ where $E_i^{(n)}=\frac{E_i^n}{[n]_i!}$ and
$F_i^{(n)}=\frac{F_i^n}{[n]_i!}$ for any $n\in \N$. Observe that if
$a_{ij}=-2$ or $a_{ij}=-3$ then $(\alpha_i,\alpha_i)=2$. Hence in
these cases one may replace $q_i$ by $q$ in the above formulas.
\subsection{Quantum symmetric pairs}
For each involutive automorphism $\theta:\gfrak\rightarrow \gfrak$ a
$q$-analog of $U(\kfrak)$ was constructed by G.~Letzter \cite{a-Letzter99a}, \cite{MSRI-Letzter} as a one-sided coideal
subalgebra $U'_q(\kfrak)$ of $U_q(\gfrak)$. Here we choose to work
with right coideal subalgebras, that is $\kow(U'_q(\kfrak))\subset
U'_q(\kfrak)\ot \uqg$. In the following sections we will give the algebra
$U'_q(\kfrak)$ as a subalgebra of $\uqg$ for each of the three example
classes (I), (II), and (III). Our conventions slightly differ from those in \cite{a-Letzter99a}, but all results from Letzter's papers translate into the present setting. In particular we will recall the presentation of
$U'_q(\kfrak)$ in terms of generators and relations following
\cite{a-Letzter03}. For each example class we will construct
the desired action of $Br(\Sigma,\theta)$ on $U'_q(\kfrak)$ by algebra
automorphism. For classical $\gfrak$, quantum analogs of $U(\kfrak)$ were previously constructed by Noumi and his coworkers \cite{a-Noumi96}, \cite{a-NS95}, \cite{a-Dijk96}. The relations between the two approaches to quantum symmetric pairs are fairly well understood \cite[Section 6]{a-Letzter99a}, \cite{a-Kolb08}.
\section{The Chevalley involution}\label{sec:Chevalley}
All through this section we consider the case where $\gfrak$ is
arbitrary but $\theta$ coincides with the Chevalley involution
$\omega$. In this case, by definition, $U'_q(\kfrak)$ is the
subalgebra of $\uqg$ generated by the elements 
\begin{align*}
  B_i = F_i -K_i^{-1}E_i \qquad \mbox{for all $i\in I$.}
\end{align*}
It follows from \eqref{eq:coproduct} that $U'_q(\kfrak)$ is a right
coideal subalgebra of $\uqg$. Up to slight conventional changes the
following result is contained in \cite[Theorem
  7.1]{a-Letzter03}. Recall that the $q$-binomial coefficient is
defined for any $i\in I$ and any $a,n\in \Z$ with $n>0$ by
\begin{align*}
  \left[\begin{matrix} a \\ n \end{matrix}\right]_i=\frac{[a]_i
    [a-1]_i \cdots [a-n+1]_i 
  }{[n]_i [n-1]_i\cdots[1]_i}.
\end{align*}
\begin{prop}\label{prop:genrels}
  The algebra $U'_q(\kfrak)$ is generated over $\qfield$ by elements
  $\{B_i\,|\,i\in I\}$ subject only to the relations 
  \begin{align*} 
    \sum_{s=0}^{1-a_{ij}}&(-1)^s
    \left[\begin{matrix}1-a_{ij}\\s \end{matrix}\right]_i
    B_i^{1-a_{ij}-s}B_j B_i^s\\  
    &=\begin{cases}
        0& \mbox{if $a_{ij}=0$,}\\
         -q_i^{-1} B_j &\mbox{if $a_{ij}=-1$,}\\
        -q^{-1}[2]^2(B_iB_j - B_j B_i) &\mbox{if $a_{ij}=-2$,}\\
        -q^{-1}([3]^2+1) (B_i^2 B_j+ B_j B_i^2) + &\\
        \qquad + q^{-1}[2]([2][4] + q^2+q^{-2}) B_i B_j B_i 
              - q^{-2} [3]^2 B_j &\mbox{if $a_{ij}=-3$.} 
      \end{cases}
  \end{align*}
\end{prop}
\subsection{Braid group action corresponding to $\omega$}\label{sec:ChevalleyBraid}
We now construct the action of $Br(\Sigma,\theta)=Br(\gfrak)$ on
$U'_q(\kfrak)$ by algebra automorphisms. For $i,j\in I$ the element
$T_i(B_j)$ does in general not belong to $\uqpk$. This was already
noted in \cite{a-MolRag08}. However, the Lusztig action still serves
as a guide to the construction of the desired braid group action on
$\uqpk$. In our conventions it is slightly easier to work with the
inverses of the Lusztig automorphisms given by \eqref{eq:L1},
\eqref{eq:L2}. The general strategy, which will also be applied to the
example classes (II) and (III) in the subsequent subsections, is a
follows. For any $i,j\in I$ we construct an element $\tau_{i}^-(B_j)$
in $\uqpk$ which coincides with $T_i^{-1}(B_j)$ up to terms of higher
weight with respect to the left adjoint action of $U^0=\qfield\langle
K_l^{\pm 1}\,|\,l\in I\rangle$ on $\uqg$. In other words,
$\tau_{i}^-(B_j)$ and $T_i^{-1}(B_j)$ have identical terms containing
maximal powers of the generators $F_l$, $l\in I$, maybe up to a
factor. For fixed $i$ we verify that the elements $\tau_{i}^-(B_j)$,
for all $j\in I$, define an algebra endomorphism $\tau_i^-$ of
$\uqpk$. An inverse is constructed using $T_i$ instead of
$T_i^{-1}$. It is then checked that the algebra automorphisms
$\tau_i^-$, $i\in I$, indeed satisfy the braid relations of
$Br(\Sigma,\theta)$. 

More precisely, for any $i,j\in I$ define
\begin{align}
  \tau_i^-(B_j)=\begin{cases}
                  B_j& \mbox{if $j{=}i$ or $a_{ij}{=}0$,}\\
                  B_i B_j-q_i B_j B_i & \mbox{if $a_{ij}=-1$,}\\
                  [2]^{-1} \big(B_i^2 B_j -q [2] B_i B_j B_i +q
                  ^2 B_j B_i^2 \big) + B_j & \mbox{if $a_{ij}=-2$,}\\
                  [3]^{-1} [2]^{-1}\big(B_i^3  B_j {-} q [3]
                  B_i^2 B_j  B_i {+} q^2  [3]  B_i B_j B_i^2  &\\  
                  \qquad\qquad\quad - q^3 B_j B_i^3 + q^{-1} (B_i
                  B_j {-} q^3 B_j B_i)\big) &\\ 
                  \qquad \qquad\qquad+ (B_i B_j - q B_j B_i)&
                  \mbox{if $a_{ij}=-3$,}\\ 
                \end{cases} \label{eq:tauimBj}
\end{align}
One calculates
\begin{align}
  T_i^{-1}(B_j)=\tau_i^-(B_j) + \epsilon(a_{ij}) 
\end{align}
where
\begin{align*}
  \epsilon(2)&= (K_j^{-1}-q_j^{-2}K_j)E_j,\\
  \epsilon(0)&=0,\\
  \epsilon(-1)&= (q_i -q_i^{-1})F_j K_i^{-1} E_i,\\
  \epsilon(-2)&=-(q-q^{-1}) (q^{-1} F_2 K_3^{-2} + (q^2-1)F_2 K_3^{-2}E_3^2 +(F_3 F_2 -q^2 F_2 F_3) K_3^{-1} E_3),\\
  \epsilon(-3)&= -(q-q^{-1})\big[(\frac{1}{[2]}F_1^2 F_2 - q^2 F_1 F_2 F_1 +\frac{q^4}{[2]}F_2 F_1^2)K_1^{-1} E_1 \\
    & \qquad \qquad \qquad + (F_1 F_2 -q^3 F_2 F_1) (q^{-1} K_1 ^{-2} + (q^2-1) K_1^{-2} E_1 ^2)\\
    & \qquad \qquad \qquad+q^{-1} (q^3-q^{-3}) F_2 K_1^{-3} E_1 +q^3 (q-q^{-1})^2 F_2 K_1^{-3} E_1^3 \big]. 
\end{align*}
The formulas for $\epsilon(-2)$ and $\epsilon(-3)$ are most easily verified by \texttt{GAP}-computations, see Subsection \ref{sec:ChevThmProof}. 
\begin{rema}
  Let $U^+$ denote the subalgebra of $\uqg$ generated by all $E_i$
  with $i\in I$. By \cite[Theorem 2.4]{a-Letzter99a} the multiplication map 
  \begin{align*}
    m:\uqpk \ot U^0 \ot U^+ \rightarrow \uqg
  \end{align*}
  is an isomorphism of vector spaces. Let $\pi^0:U^0\rightarrow
  \qfield$ map any Laurent polynomial in $U^0=\qfield\langle K_i,
  K_i^{-1}\,|\,i \in I\rangle$ to its constant term. One defines a
  projection of vector spaces $\pi: \uqg \rightarrow \uqpk$ by
  $\pi(u)=(\id\ot \pi^0\ot \vep)\circ m^{-1}(u)$. It follows from the
  above formulas that 
  \begin{align*}
    \tau_i^-(B_j) &= \pi\circ T_i^{-1} (B_j).
  \end{align*} 
However, the projection map $\pi$  is no algebra
homomorphism. Therefore it is a priori unclear that $\tau_i^-$ is an
algebra homomorphism. Nevertheless, this holds by the following Theorem. 
\end{rema}
\begin{thm} \label{thm:braidChev}  Let $i\in I$.\\
  1) There exists a unique algebra automorphism $\tau_i^-$ of $\uqpk$
  such that $\tau^-_i(B_j)$ is given by \eqref{eq:tauimBj}.\\ 
  2) The inverse automorphism $\tau_i$ of $\tau_i^-$ is determined by 
\begin{align} 
  \tau_i(B_j)=\begin{cases}
                  B_j& \mbox{if $j{=}i$ or $a_{ij}{=}0$,}\\
                  B_j B_i - q_i B_i B_j & \mbox{if $a_{ij}=-1$,}\\
                  [2]^{-1} \big(B_j B_i^2 -q [2] B_i B_j B_i +q^2
                  B_i^2 B_j \big) + B_j & \mbox{if $a_{ij}=-2$,}\\ 
                  [3]^{-1} [2]^{-1}\big(B_j B_i^3  {-} q [3]
                  B_i B_j  B_i^2 {+} q^2  [3]  B_i^2 B_j B_i  &\\  
                  \qquad\qquad\quad - q^3  B_i^3B_j + q^{-1} (B_j
                  B_i {-} q^3 B_i B_j)\big) &\\ 
                  \qquad \qquad\qquad+ (B_j B_i - q B_i B_j)&
                  \mbox{if $a_{ij}=-3$.}\\ 
                \end{cases} \label{eq:tauiBj}
\end{align}
3) There exists a unique group homomorphism $Br(\gfrak)\rightarrow
\Aut_{alg}(\uqpk)$ such that $s_j\mapsto \tau_j^-$ for all $j\in I$.
\end{thm}
The proof of the theorem is given by direct computations
using the computer algebra package \texttt{QuaGroup} \cite{a-quagroup} within
\texttt{GAP} \cite{GAP4} for calculations with quantum enveloping algebras. More details will be given in Subsection \ref{sec:ChevThmProof}. For
$\gfrak$ of type $ADE$, however, the statement of Theorem
\ref{thm:braidChev} follows from results in \cite{a-MolRag08} as
explained in the  following two remarks. 
\begin{rema} \label{rem:A-Chev} Assume that $\gfrak=\slfrak_n(\C)$. We
  want to relate the above theorem to the braid group action
  constructed in \cite{a-MolRag08}. To this end define $S_i=
  (q-q^{-1})B_i$ for $i=1,\dots,N-1$ and observe that by Proposition
  \ref{prop:genrels} the algebra $\uqpk$ is generated by the elements
  $S_i$ subject only to the relations given in \cite[above Theorem 
    2.1]{a-MolRag08}. Hence, in this case, $\uqpk$ coincides with the 
  algebra $U'_q(\ofrak)$ considered in \cite{a-MolRag08} and
  originally introduced by Gavrilik and Klimyk in
  \cite{a-GavKlimyk91}. By \cite[Theorem 2.1]{a-MolRag08}, for any
  $i=1,\dots,N-1$, there exists an automorphism $\beta_i$ of $\uqpk$ such that
  \begin{align*}
    \beta_i(B_j)=\begin{cases}
                   -[B_i,B_j]_q & \mbox{if $j=i+1$,}\\
                    [B_i,B_j]_q & \mbox{if $j=i-1$,}\\
                    -B_i & \mbox{if $i=j$,}\\
                     B_j &\mbox{else}
                 \end{cases}
  \end{align*}
  where by definition $[a,b]_q=ab-qba$ for any $a,b\in \uqg$.
  Now we observe that $\tau_i^-=\beta_i\circ\kappa_i$
  where $\kappa_i:\uqpk\rightarrow \uqpk$ is the algebra automorphism defined by
  \begin{align*}
    \kappa_i(B_j)= \begin{cases}
                -B_j&\mbox{if $j=i$ or $j=i+1$,}\\
                 B_j&\mbox{else.}
               \end{cases}
  \end{align*}
  In view of commutation relations
  \begin{align*}
    \kappa_{i+1}\circ\beta_i=\beta_i \circ\kappa_{i+1},\quad 
      \kappa_{i+1}\circ\beta_{i+1}=\beta_{i+1} \circ\kappa_i,\\
    \kappa_{i}\circ\beta_{i+1}=\beta_{i+1}
    \circ\kappa_{i+1}\circ\kappa_{i+2}\circ\dots\circ
    \kappa_{N-1}
  \end{align*}
  the braid relations for the automorphisms $\{\beta_i\,|\,i\in I\}$
  are equivalent to the braid relations for the automorphisms
  $\{\tau_i^-\,|\,i\in I\}$. Hence the statements of Theorem
  \ref{thm:braidChev} 1) and 3) for $\gfrak=\slfrak_N(\C)$ are equivalent to
  \cite[Theorem 2.1]{a-MolRag08} the proof of which also contains a
  formula for the inverse of $\beta_i$.    
\end{rema} 
\begin{rema}\label{rem:DE-Chev}
  Assume now that $\gfrak$ is of type $D_n$ or $E_n$. The fact that
  $\tau_i^-$ is an algebra endomorphism of $\uqpk$ follows from the
  corresponding fact for $\gfrak$ of type $A_n$ because the elements 
  $\tau_i^-(B_j)$ and $\tau_i^-(B_k)$ are contained in the subalgebra
  of $\uqg$ corresponding the subset $\{i,j,k\}$ of $I$. The same
  holds for the inverse $\tau_i$ and therefore $\tau_i^-$ is an
  algebra automorphism. Each side of a braid relation evaluated on
  any generator $B_k$ is again contained in the subalgebra of $\uqg$
  corresponding to a subset of $I$ with at most three elements. Hence
  the braid relations for $\gfrak$ of type $A_n$ also imply the braid
  relations for $\{\tau_i^-\,|\,i\in I\}$ if $\gfrak$ is of type $D_n$
  or $E_n$.  
\end{rema}
\subsection{The proof of Theorem \ref{thm:braidChev}}\label{sec:ChevThmProof}
In this subsection we explain how to verify Theorem \ref{thm:braidChev} using the package \texttt{QuaGroup} under \texttt{GAP}. In view of Remarks \ref{rem:A-Chev} and \ref{rem:DE-Chev} it remains to consider the cases where $\gfrak$ is of type $B_n, C_n, F_4$, or $G_2$. By Theorem \ref{thm:braidChev} for
type $A_n$ it even suffices to consider the cases where $\gfrak$
is of type $B_3, C_3, F_4$, or $G_2$. Moreover, the claim of Theorem
\ref{thm:braidChev} for $\gfrak$ of type $F_4$ will follow from the
corresponding claims for $\gfrak$ of type $B_3$ and $C_3$. Hence we
only need to consider the three remaining cases $B_3$, $C_3$, and $G_2$.

In the cases $B_3$ and $C_3$ this is done by the \texttt{GAP} codes \texttt{I-B3.txt} and
\texttt{I-C3.txt} which are available from \cite{WWW}. In each code the generators of $\uqpk$ are defined and it is checked that the relations of Proposition \ref{prop:genrels} are satisfied. Next it is verified that for fixed $i$ the images $\tau_i^-(B_j)$ and $\tau_i(B_j)$ also satisfy the relations of Proposition \ref{prop:genrels}. This proves that Equations \eqref{eq:tauimBj} and \eqref{eq:tauiBj} give well-defined algebra endomorphisms of $\uqpk$. It is then checked that $\tau_i^-\circ \tau_i(B_j)=B_j=\tau_i\circ\tau_i^-(B_j)$ for all $i,j$ which implies that $\tau_i$ and $\tau_i^-$ are mutually inverse and hence algebra automorphisms. Finally, the braid relations are verified when evaluated on the generators. This completes the proof of Theorem \ref{thm:braidChev} in these cases.

In the case $G_2$ parts 1) and 2) of Theorem \ref{thm:braidChev} are verified in the same way as described above using the \texttt{GAP} code \texttt{I-G2.txt} \cite{WWW}. However, due to memory problems, the \texttt{GAP} code provided in \texttt{I-G2.txt} crashes when it tries to verify the $G_2$ braid relations. Istvan Heckenberger kindly checked the $G_2$ braid relations for us, using the noncommutative algebra program FELIX \cite{a-FELIX}. His code \texttt{G2-Braid.flx} and his output file \texttt{G2-Braid.aus} are also available from \cite{WWW}. In these calculations it turned out that
$(\tau_1 \circ\tau_2)^3=\id_{\uqpk}=(\tau_2 \circ \tau_1)^3$ if $\gfrak$ is of type $G_2$. Analogously, one has $(\tau_1 \circ\tau_2)^2=\id_{\uqpk}=(\tau_2 \circ \tau_1)^2$ if $\gfrak$ is of type $B_2$, however not in higher rank.

\section{The involutive automorphism $\theta=\tau\circ\omega$}\label{sec:invol}
In this section we consider the case (II), that is $\gfrak$ is of type $A_n$, $D_n$, or $E_6$ and $\theta=\tau\circ\omega$ where $\tau$ is the nontrivial diagram automorphism of order 2. In this case, by definition, $\uqpk$ is the subalgebra of $\uqg$ generated by the elements
\begin{align}\label{eq:invol-generators}
  B_i = F_i -K_i^{-1} E_{\tau(i)}, \qquad K_i K_{\tau(i)}^{-1}
\end{align}
for all $i\in I$. Again it follows from \eqref{eq:coproduct} that $\uqpk$ is a right coideal subalgebra of $\uqg$. Let $\qfield T_\theta$ denote the subalgebra of $\uqpk$ generated by the elements $K_i K_{\tau(i)}^{-1}$ for all $i\in I$ and observe that $\qfield T_\theta$ is a Laurent polynomial ring. The following result is contained in \cite[Theorem 7.1]{a-Letzter03}.
\begin{prop}\label{prop:genrels-tauom}
  The algebra $\uqpk$ is generated over $\qfield T_\theta$ by elements $\{B_i\,|\,i\in I\}$ subject only to the relations 
  \begin{align}
    K_i K_{\tau(i)}^{-1} B_j &= q^{(\alpha_j, \alpha_{\tau(i)}-\alpha_i)}B_j K_i K_{\tau(i)}^{-1}& \mbox{for all $i,j\in I$},\nonumber\\
    B_i B_j - B_j B_i &= \delta_{\tau(i), j} \frac{K_i K_{\tau(i)}^{-1} - K_{\tau(i)} K_i ^{-1}}{q-q^{-1}}& \mbox{if $a_{ij}=0$,}  \label{eq:com-rel}\\
    B_i^2 B_j -(q+q^{-1}&) B_i B_j B_i + B_j B_i^2 = - \delta_{i,\tau(i)} q^{-1} B_j -& \nonumber\\
    - \delta_{j,\tau(i)}(q&+q^{-1})B_i (q^{-1}K_i K_{\tau(i)}^{-1}+ q^2 K_{\tau(i)}K_i^{-1})  & \mbox{if $a_{ij}=-1$.}\nonumber
  \end{align}
\end{prop}
\begin{rema}
  In the case (IIA) with $n=2r-1$ the theory of quantum symmetric pairs actually provides a family of coideal subalgebras $\uqpks$ depending on a parameter $s\in \qfield$. By definition, $\uqpks$ is the subalgebra generated by the elements \eqref{eq:invol-generators} for $i\neq r$ and by $B_r=F_r - K_r^{-1}E_r + s K_r^{-1}$. By \cite[Theorem 7.1]{a-Letzter03}, however, the $\uqpks$ are pairwise isomorphic as algebras for different parameters $s$. For the purpose of constructing an action of $Br(\gfrak,\theta)$ on $\uqpks$ by algebra automorphisms it hence suffices to consider the case $s=0$ only.
\end{rema}
The case $a_{ij}=-1$ and $j=\tau(i)$, which leads to an additional summand in the last relation in Proposition \ref{prop:genrels-tauom}, can only occur if $\gfrak$ is of type $A_n$ with even $n$. For simplicity we first exclude this case in the following subsection. It will be treated in Subsection \ref{sec:IIAeven}. Subsections \ref{sec:IID} and \ref{sec:IIE} are devoted to the braid group actions on $\uqpk$ in the cases (IID) and (IIE), respectively. All theorems in the present section are verified by \texttt{GAP} calculations. More details and references to the \texttt{GAP}-codes are given in Subsection \ref{sec:invol-GAP}.
\subsection{The braid group action in the case (IIA) for odd $n$}\label{sec:IIAodd}
Throughout this subsection we consider the case (IIA) with $n=2r-1$. By Proposition \ref{prop:BrSinBrTheta} and Corollary \ref{cor:Bract} we aim to construct an action of the braid group $Br(\bfrak_r)$ on $\uqpk$ by algebra automorphisms. Hence we need to construct a family of algebra automorphisms $\{\tau_1,\dots,\tau_r\}$ of $\uqpk$ which satisfy the type $B_r$ braid relations. Again, the construction of the $\tau_i$ is guided by the Lusztig automorphisms of $\uqg$. For $i=1,\dots,r-1$ and $x\in \qfield T_\theta$ define elements $\tau_i(x),\tau_i^-(x)\in \qfield T_\theta$ by
\begin{align}
  \tau_i(x)&=\tau_i^-(x)=T_i T_{\tau(i)}(x)\label{eq:tauidef}
\end{align}
and define
\begin{align}
  \tau_r(x)&=\tau_r^-(x)=T_r(x).  \label{eq:taurdef}
\end{align}
Observe that $\tau_r(x)=x$ for all $x\in \qfield T_\theta$.
It remains to define the action of $\{\tau_i\,|\,i=1,\dots,r\}$ on the generators $\{B_j\,|\,j=1,\dots,n\}$ of $\uqpk$. The commutator relations \eqref{eq:com-rel} together with \eqref{eq:tauidef} already impose a significant restriction. For $1\le i\le r-1$ and $j=i-1$ one can check that 
\begin{align*}
  [B_i,B_j]_q [B_{\tau(i)},B_{\tau(j)}]_q - [B_{\tau(i)},B_{\tau(j)}]_q  [B_i,B_j]_q = q\, \frac{\tau_i^{-}(K_j K_{\tau(j)}^{-1}) - \tau_j^-(K_{\tau(j)} K_j^{-1})}{q-q^{-1}}.
\end{align*}
In view of Equation \eqref{eq:com-rel} the additional factor $q$ on the right hand side explains the appearance of fractional powers of $q$ in the following definition. For $1\le i\le r-1$ define 
\begin{align}\label{eq:IIAtauimdef}
  \tau_i^-(B_j)&=\begin{cases}
                   q^{-1/2} [B_i,B_j]_q & \mbox{if $a_{ij}=-1$ and $a_{\tau(i)j}\neq-1$,}\\
                   q^{-1/2} [B_{\tau(i)},B_j]_q & \mbox{if $a_{ij}\neq-1$ and $a_{\tau(i)j}=-1$,}\\
                   q^{-1} [B_i,[B_{\tau(i)},B_j]_q]_q + B_j K_i K_{\tau(i)}^{-1} & \mbox{if $a_{ij}=-1$ and $a_{\tau(i)j}=-1$,}\\
                   q^{-1} K_i K_{\tau(i)}^{-1} B_{\tau(i)} &\mbox{if $j=i$,}\\
                   q^{-1} K_{\tau(i)} K_i^{-1} B_i &\mbox{if $j=\tau(i)$,}\\
                   B_j & \mbox{else.}
                 \end{cases} 
\end{align}
Observe that the case $a_{ij}=-1$ and $a_{\tau(i)j}=-1$ only occurs for $i=r-1$ and $j=r$.
\begin{thm}\label{thm:tauim}
  Let $1\le i \le r-1$.\\
  1) There exists a unique algebra automorphism $\tau_i^-$ of $\uqpk$ such that $\tau_i^-(B_j)$ is given by \eqref{eq:IIAtauimdef} for $j=1,\dots,n$, and $\tau_i^-(x)$ for $x\in \qfield T_\theta$ is given by \eqref{eq:tauidef}. \\
  2) The inverse automorphism $\tau_i$ of $\tau_i^-$ is determined by \eqref{eq:tauidef} and by
\begin{align}
  \tau_i(B_j)&=\begin{cases}
                   q^{-1/2} [B_j,B_i]_q & \mbox{if $a_{ij}=-1$ and $a_{\tau(i)j}\neq-1$,}\\
                   q^{-1/2} [B_j,B_{\tau(i)}]_q & \mbox{if $a_{ij}\neq-1$ and $a_{\tau(i)j}=-1$,}\\
                   q^{-1} [[B_j,B_i]_q, B_{\tau(i)}]_q + B_j K_i K_{\tau(i)}^{-1} & \mbox{if $a_{ij}=-1$ and $a_{\tau(i)j}=-1$,}\\
                   q K_{\tau(i)} K_i^{-1} B_{\tau(i)} &\mbox{if $j=i$,}\\
                   q K_i K_{\tau(i)}^{-1} B_i &\mbox{if $j=\tau(i)$,}\\
                   B_j & \mbox{else.}
                 \end{cases} 
\end{align}    
3) The relation $\tau_{i-1} \tau_i \tau_{i-1}= \tau_i \tau_{i-1} \tau_i$ holds if $2\le i\le r-1$. Moreover, $\tau_i \tau_j =\tau_j \tau_i$ if $|i-j|\neq 1$.
\end{thm}
It remains to construct the algebra automorphisms $\tau_r^-$ and $\tau_r$. To this end define
\begin{align}\label{eq:IIAtaurmdef}
  \tau_r^-(B_j)=\begin{cases}
                  [B_r,B_j]_q & \mbox{if $j=r-1$ or $j=r+1$,}\\
                  B_j & \mbox{else.} 
                \end{cases}
\end{align}
\begin{thm}\label{thm:taurm}
  1) There exists a unique algebra automorphism $\tau_r^-$ of $\uqpk$ such that $\tau_r^-(B_j)$ is given by \eqref{eq:IIAtaurmdef} for $j=1,\dots,n$, and $\tau_r^-(x)$ for $x\in \qfield T_\theta$ is given by \eqref{eq:taurdef}. \\
  2) The inverse automorphism $\tau_r$ of $\tau_r^-$ is determined by \eqref{eq:tauidef} and by
\begin{align}
   \tau_r(B_j)=\begin{cases}
                  [B_j,B_r]_q & \mbox{if $j=r-1$ or $j=r+1$,}\\
                   B_j & \mbox{else.} 
                \end{cases}
\end{align}    
3) The relation $\tau_{-1} \tau_r \tau_{r-1} \tau_r= \tau_r \tau_{r-1}  \tau_r  \tau_{r-1}$ holds. Moreover, $\tau_r \tau_j =\tau_j \tau_r$ if $j=1,\dots,r-2$.
\end{thm}
Summarizing the two above theorems one obtains the following result.
\begin{cor}\label{cor:IIABraidAction}
  There exists a unique group homomorphism 
  \begin{align*}
    Br(\bfrak_r)\rightarrow \Aut_{alg}(\uqpk)
  \end{align*}
  such that $\os_i\mapsto \tau_i$ for all $i=1,\dots,r$.
\end{cor}
\subsection{The braid group action in the case (IIA) for even $n$}\label{sec:IIAeven}
We now turn to the case (IIA) with $n=2r$. Motivated by Proposition \ref{prop:BrSinBrTheta} and Corollary \ref{cor:Bract} we again aim to construct an action of the braid group $Br(\bfrak_r)$ on $\uqpk$ by algebra automorphisms. As in \eqref{eq:tauidef} and \eqref{eq:taurdef} the restriction of this braid group action to the subalgebra $\qfield T_\theta$ of $\uqpk$ is determined by the Lusztig automorphisms. However, the difference between the cases $n=2r-1$ and $n=2r$ in Proposition \ref{prop:BrSinBrTheta} needs to be taken into account. More explicitly, for $i=1,\dots, r-1$ and $x\in \qfield T_\theta$ define
\begin{align}
  \tau_i(x)&=\tau_i^-(x)=T_i T_{\tau(i)}(x),\label{eq:tauidef2r}\\
  \tau_r(x)&=\tau_r^-(x)=T_r T_{r-1} T_r(x).  \label{eq:taurdef2r}
\end{align}
For $1\le i\le r-1$ the automorphisms $\tau^-_i$ of $\uqpk$ can be defined as in the previous subsection. They get simpler because the case ($a_{ij}=-1$ and $a_{\tau(i)j}=-1$) cannot occur if $n=2r$. Hence, for $1\le i\le r-1$ one defines
\begin{align}\label{eq:IIAtauimdef2r}
  \tau_i^-(B_j)&=\begin{cases}
                   q^{-1/2} [B_i,B_j]_q & \mbox{if $a_{ij}=-1$,}\\
                   q^{-1/2} [B_{\tau(i)},B_j]_q & \mbox{if $a_{\tau(i)j}=-1$,}\\
                   q^{-1} K_i K_{\tau(i)}^{-1} B_{\tau(i)} &\mbox{if $j=i$,}\\
                   q^{-1} K_{\tau(i)} K_i^{-1} B_i &\mbox{if $j=\tau(i)$,}\\
                   B_j & \mbox{else.}
                 \end{cases} 
\end{align}
One obtains the following analog of Theorem \ref{thm:tauim}.
\begin{thm}\label{thm:tauimthm2r}
  Let $1\le i \le r-1$.\\
  1) There exists a unique algebra automorphism $\tau_i^-$ of $\uqpk$ such that $\tau_i^-(B_j)$ is given by \eqref{eq:IIAtauimdef2r} for $j=1,\dots,n$, and $\tau_i^-(x)$ for $x\in \qfield T_\theta$ is given by \eqref{eq:tauidef2r}. \\
  2) The inverse automorphism $\tau_i$ of $\tau_i^-$ is determined by \eqref{eq:tauidef2r} and by
\begin{align}
  \tau_i(B_j)&=\begin{cases}
                   q^{-1/2} [B_j,B_i]_q & \mbox{if $a_{ij}=-1$,}\\
                   q^{-1/2} [B_j,B_{\tau(i)}]_q & \mbox{if $a_{\tau(i)j}=-1$,}\\
                   q K_{\tau(i)} K_i^{-1} B_{\tau(i)} &\mbox{if $j=i$,}\\
                   q K_i K_{\tau(i)}^{-1} B_i &\mbox{if $j=\tau(i)$,}\\
                   B_j & \mbox{else.}
                 \end{cases} 
\end{align}    
3) The relation $\tau_{i-1} \tau_i \tau_{i-1}= \tau_i \tau_{i-1} \tau_i$ holds if $2\le i\le r-1$. Moreover, $\tau_i \tau_j =\tau_j \tau_i$ if $|i-j|\neq 1$.
\end{thm}
Again it remains to construct the algebra automorphisms $\tau^-_r$ and $\tau_r$. Here the difference between the cases $n=2r-1$ and $n=2r$ in Proposition \ref{prop:BrSinBrTheta} is significant. Define
\begin{align}\label{eq:IIAtaurmdef2r}
  \tau_r^-(B_j)=\begin{cases}
                  q^{-3/2}[B_{r+1},[B_r,B_{r-1}]_q]_q + q^{1/2} K_r^{-1}K_{r+1} B_{r-1} & \mbox{if $j=r-1$,}\\
                  q^{-3/2} K_r^{-1}K_{r+1} B_r & \mbox{if $j=r$,}\\
                  q^{-3/2} K_r K_{r+1}^{-1} B_{r+1} & \mbox{if $j=r+1$,}\\
                  q^{-3/2}[B_{r},[B_{r+1},B_{r+2}]_q]_q + q^{1/2} K_r K_{r+1}^{-1} B_{r+2} & \mbox{if $j=r+2$,}\\
                  B_j & \mbox{else.} 
                \end{cases}
\end{align}
\begin{thm}
  1) There exists a unique algebra automorphism $\tau_r^-$ of $\uqpk$ such that $\tau_r^-(B_j)$ is given by \eqref{eq:IIAtaurmdef2r} for $j=1,\dots,n$, and $\tau_r^-(x)$ for $x\in \qfield T_\theta$ is given by \eqref{eq:taurdef2r}. \\
  2) The inverse automorphism $\tau_r$ of $\tau_r^-$ is determined by \eqref{eq:tauidef2r} and by
\begin{align}
   \tau_r(B_j)=\begin{cases}
                  q^{-3/2}[[B_{r-1}, B_r]_q,B_{r+1}]_q + q^{-1/2} K_r K_{r+1}^{-1} B_{r-1} & \mbox{if $j=r-1$,}\\
                  q^{3/2} K_r K_{r+1}^{-1} B_r & \mbox{if $j=r$,}\\
                  q^{3/2} K_r^{-1} K_{r+1} B_{r+1} & \mbox{if $j=r+1$,}\\
                  q^{-3/2}[[B_{r+2}, B_{r+1}]_q,B_r]_q+ q^{-1/2} K_r^{-1} K_{r+1} B_{r+2} & \mbox{if $j=r+2$,}\\
                  B_j & \mbox{else.} 
                \end{cases}
\end{align}    
3) The relation $\tau_{r-1} \tau_r \tau_{r-1} \tau_r= \tau_r \tau_{r-1}  \tau_r  \tau_{r-1}$ holds. Moreover, $\tau_r \tau_j =\tau_j \tau_r$ if $j=1,\dots,r-2$.
\end{thm}
Observe that Corollary \ref{cor:IIABraidAction} holds literally in the setting of the present subsection.
\subsection{The braid group action in the case (IID)}\label{sec:IID}
The braid group action on $\uqpk$ in the cases (IID) and (IIE) are obtained as a combination of the results in the case (IIA) with $n=2r-1$ odd with the results of Subsection \ref{sec:Chevalley}. We first consider the case (IID), that is, $\gfrak$ is of type $D_{n+1}$ and $\uqpk$ is the subalgebra of $\uqg$ generated by the elements
\begin{align*}
  B_i &= F_i - K_i^{-1} E_i \qquad \mbox{ for $i=1, \dots, n-1$,}\\
  B_n &= F_n - K_n^{-1} E_{n+1},\\
  B_{n+1} &= F_{n+1} - K_{n+1}^{-1} E_n,\\
  K_n &K_{n+1}^{-1},\,\, K_n^{-1} K_{n+1}.
\end{align*}
Following Proposition \ref{prop:BrSinBrTheta} and Corollary \ref{cor:Bract} we aim to find an action of $Br(\bfrak_n)$ on $\uqpk$ by algebra automorphisms. For $i=1,\dots, n-1$, following the constructions from Subsections \ref{sec:ChevalleyBraid} and \ref{sec:IIAodd}, define 
\begin{align}\label{eq:tauimBIID}
  \tau_i^-(B_j)=\begin{cases}
                 [B_i,B_j]_q & \mbox{if $a_{ij}=-1$,}\\
                 B_j & \mbox{else.}
               \end{cases}
\end{align}
and for $x\in \qfield T_\theta$ define moreover elements $\tau_i(x), \tau_i^-(x)\in \qfield T_\theta$ by
\begin{align}\label{eq:tauimxIID}
  \tau_i^-(x)= \tau_i(x)= T_i(x).
\end{align}
One obtains the following analog of Theorems \ref{thm:braidChev} and \ref{thm:taurm}.
\begin{thm}\label{thm:tauimIID}
  Let $1\le i \le n-1$.\\
  1) There exists a unique algebra automorphism $\tau_i^-$ of $\uqpk$ such that $\tau_i^-(B_j)$ is given by \eqref{eq:tauimBIID} for $j=1,\dots,n+1$, and $\tau_i^-(x)$ for $x\in \qfield T_\theta$ is given by \eqref{eq:tauimxIID}. \\
  2) The inverse automorphism $\tau_i$ of $\tau_i^-$ is determined by \eqref{eq:tauimxIID} and by
\begin{align}
  \tau_i(B_j)=\begin{cases}
                 [B_j,B_i]_q & \mbox{if $a_{ij}=-1$,}\\
                 B_j & \mbox{else.}
               \end{cases}
\end{align}    
3) The relation $\tau_{i-1} \tau_i \tau_{i-1}= \tau_i \tau_{i-1} \tau_i$ holds if $2\le i\le n-1$. Moreover, $\tau_i \tau_j =\tau_j \tau_i$ if $|i-j|\neq 1$ and $1\le j \le n-1$.
\end{thm}
Now we follow \eqref{eq:IIAtauimdef} and define
\begin{align}\label{eq:IIDtaunmdef}
  \tau_n^-(B_j)=\begin{cases}
                   q^{-1} [B_n,[B_{n+1},B_{n-1}]_q]_q + B_{n-1} K_n K_{n+1}^{-1}
                                                & \mbox{if $j=n-1$,}\\
                   q^{-1} K_n K_{n+1}^{-1} B_{n+1} &\mbox{if $j=n$,}\\
                   q^{-1} K_{n+1} K_n^{-1} B_n &\mbox{if $j=n+1$,}\\
                  B_j & \mbox{else}
                \end{cases}
\end{align}
and
\begin{align}\label{eq:taundefIID}
  \tau_n^-(x)=\tau_n(x)= T_n T_{n+1}(x)
\end{align}
for all $x\in \qfield T_\theta$.
\begin{thm}\label{thm:taunmIID}
   1) There exists a unique algebra automorphism $\tau_n^-$ of $\uqpk$ such that $\tau_n^-(B_j)$ is given by \eqref{eq:IIDtaunmdef} for $j=1,\dots,n+1$, and $\tau_n^-(x)$ for $x\in \qfield T_\theta$ is given by \eqref{eq:taundefIID}. \\
  2) The inverse automorphism $\tau_n$ of $\tau_n^-$ is determined by \eqref{eq:taundefIID} and by
\begin{align}
  \tau_n(B_j)=\begin{cases}
                   q^{-1} [[B_{n-1},B_n]_q, B_{n+1}]_q + B_{n-1} K_n K_{n+1}^{-1}
                                                & \mbox{if $j=n-1$,}\\
                   q K_{n+1} K_{n}^{-1} B_{n+1} &\mbox{if $j=n$,}\\
                   q K_{n} K_{n+1}^{-1} B_n &\mbox{if $j=n+1$,}\\
                  B_j & \mbox{else.}
                \end{cases}
\end{align}    
3) The relation $\tau_{n-1} \tau_n \tau_{n-1} \tau_n= \tau_n \tau_{n-1}  \tau_n  \tau_{n-1}$ holds. Moreover, $\tau_n \tau_j =\tau_j \tau_n$ if $j=1,\dots,n-2$.
\end{thm}
Theorems \ref{thm:tauimIID} and \ref{thm:taunmIID} yield an action of the braid group $Br(\bfrak_n)$ on $\uqpk$ by algebra automorphisms such that $\os_i\mapsto \tau_i$ for $i=1,\dots,n$.
\subsection{The braid group action in the case (IIE)}\label{sec:IIE}
We now turn to the case (IIE), that is, $\gfrak$ is of type $E_6$ and $\uqpk$ is the subalgebra of $\uqg$ generated by the elements
\begin{align*}
  B_1 &= F_1 - K_1^{-1} E_6,& B_4 &= F_4 - K_4^{-1} E_4,\\
  B_2 &= F_2 - K_2^{-1} E_2,& B_5 &= F_5 - K_5^{-1} E_3,\\
  B_3 &= F_3 - K_3^{-1} E_5,& B_6 &= F_6 - K_6^{-1} E_1,\\
  K_1 &K_{6}^{-1} , K_6^{-1} K_{1}, & K_3 &K_{5}^{-1} , K_5^{-1} K_{3}.
\end{align*}
By Proposition \ref{prop:BrSinBrTheta} and Corollary \ref{cor:Bract} we expect to find an action of $Br(\ffrak_4)$ on $\uqpk$ by algebra automorphisms. For $x\in \qfield T_\theta$ and $j=1,2,\dots,6$ define
\begin{align}
   \tau_1(x)&= T_1 T_6 (x),&
  \tau_1(B_j)&=\begin{cases}
                 q K_6 K_1^{-1} B_6& \mbox{if $j=1$,}\\
                 B_j & \mbox{if $j=2,4$,}\\
                 q^{-1/2} [B_3,B_1]_q & \mbox{if $j=3$,}\\
                 q^{-1/2} [B_5,B_6]_q & \mbox{if $j=5$,}\\
                 q K_1 K_6^{-1} B_1& \mbox{if $j=6$,}
               \end{cases} \label{eq:E6tau1}\\
  \tau_2(x)&= T_3 T_5 (x),&             
  \tau_2(B_j)&=\begin{cases}
                 q^{-1/2} [B_1,B_3]_q & \mbox{if $j=1$,}\\
                 B_2 & \mbox{if $j=2$,}\\
                 q K_5 K_3^{-1} B_5& \mbox{if $j=3$,}\\
               q^{-1} [[B_4,B_3]_q, B_5]_q + B_{4} K_3 K_{5}^{-1}
                                                & \mbox{if $j=4$,}\\
                 q K_3 K_5^{-1} B_3& \mbox{if $j=5$,}\\
                 q^{-1/2} [B_6,B_5]_q & \mbox{if $j=6$,}
               \end{cases}\label{eq:E6tau2}\\
  \tau_3(x)&= T_4(x),&
  \tau_3(B_j)&=\begin{cases}
                 B_j& \mbox{if $j=1,4,6$,}\\
                 [B_j,B_4]_q& \mbox{if $j=2,3,5$,}\\
               \end{cases}\label{eq:E6tau3}\\              
   \tau_4(x)&= T_2(x),&
  \tau_4(B_j)&=\begin{cases}
                 B_j& \mbox{if $j=1,2,3,5,6$,}\\
                 [B_3,B_2]& \mbox{if $j=4$.}
               \end{cases}\label{eq:E6tau4}                                      
\end{align}
The next theorem is implied by the corresponding results in the cases (IIA) for odd $n$ and (IID) from Subsections \ref{sec:IIAodd} and \ref{sec:IID}, respectively.
\begin{thm}
  1) There exist uniquely determined algebra automorphisms $\tau_i$ of $\uqpk$ for $i=1,\dots, 4$ such that $\tau_i(B_j)$ and $\tau_i(x)$  are given by \eqref{eq:E6tau1} -- \eqref{eq:E6tau4} for $j=1,\dots,6$ and $x\in \qfield T_\theta$. \\
  2) There exists a unique group homomorphism $Br(\ffrak_4)\rightarrow \Aut_{alg}(\uqpk)$ such that $\os_i\mapsto \tau_i$ for all $i=1,\dots,4$.
\end{thm}
The inverse automorphisms $\tau_i^-$ of $\tau_i$ for $i=1,\dots,4$ can also be read off the corresponding formulas in Subsections \ref{sec:IIAodd} and \ref{sec:IID}.
\subsection{The proofs of the theorems in Section \ref{sec:invol}}\label{sec:invol-GAP}
The proofs of the theorems in Subsections \ref{sec:IIAodd}, \ref{sec:IIAeven}, and \ref{sec:IID} are again performed via \texttt{GAP} using the codes
\texttt{II-A7.txt}, \texttt{II-A6.txt}, and \texttt{II-D5.txt}, respectively, which are available from \cite{WWW}. 

We first turn to the case (IIA) for odd $n$ considered in Subsection \ref{sec:IIAodd}. Observe that it suffices to prove Theorems \ref{thm:tauim} and \ref{thm:taurm} in the case $r=4$, that is for $n=7$. Indeed, if in this case $\tau_2$ and $\tau_2^-$ are mutually inverse algebra automorphisms of $\uqpk$ then $\tau_i$ and $\tau_i^-$ are mutually inverse algebra automorphisms of $\uqpk$ for $i=1,\dots, r-2$ in the case of arbitrary $r$. Similarly, the fact that for $r=4$ the maps $\tau_3$ and $\tau_4$ are algebra automorphisms with inverses $\tau_3^-$ and $\tau_4^-$, respectively, implies that for general $r$ the maps  $\tau_{r-1}$ and $\tau_r$ are algebra automorphisms with inverses $\tau_{r-1}^-$ and $\tau_r^-$, respectively. Finally, the braid relations between the $\tau_i$ for $r=4$ imply the braid relations between the $\tau_i$ for general $r$. With the \texttt{GAP}-code \texttt{II-A7.txt} one checks all the relations necessary to prove Theorems \ref{thm:tauim} and \ref{thm:taurm} by a method similar to the one described in Subsection \ref{sec:ChevThmProof} for case I.

In the case (IIA) for even $n$ one sees by reasoning similar to the one above that it is sufficient to prove the theorems of Subsection \ref{sec:IIAeven} in the case $r=3$, that is $n=6$. With the \texttt{GAP}-code \texttt{II-A6.txt} one may check all the necessary relations in this case.
Similarly, it suffices to consider the case that $\gfrak$ is of type $D_5$ in order to verify the results of Subsection \ref{sec:IID}. The \texttt{GAP}-code \texttt{II-D5.txt} performs all the necessary checks.

\section{The involutive automorphism $\Ad(w_X)\circ \omega$}\label{sec:wXom}
In this section we consider the case (III), that is, $\gfrak=\slfrak_n(\C)$ with $n=2m$ even, and $\theta=\Ad(w_X)\circ\omega$ where $w_X=s_1 s_3 s_5 \cdots s_{2m-1}$.
In this case, by definition, $\uqpk$ is the subalgebra of $U_q(\slfrak_n(\C))$,
generated by the elements 
\begin{align}
  &E_{i}, F_{i}, K_{i}^{\pm 1}&&\mbox{for $i$
    odd,}\label{odd}\\
  &F_{i}- K_{i}^{-1}\ad(E_{i-1}E_{i+1})(E_{i})&&
  \mbox{for $i$ even.}\label{even}
\end{align}
Here $\ad(x)(u)=\sum_r x_{r}u S(x'_{r})$ denotes the adjoint action
of $x$ on $u$ for $u,x\in \uqg$ with $\kow(x)=\sum_r x_r\ot x_r'$, see \cite[4.18]{b-Jantzen96}.
Again one checks that $\uqpk$ is a right coideal subalgebra of
$U_q(\slfrak_n(\C))$. We define
\begin{align*}
  B_i=\begin{cases}
        F_{i}& \mbox{ if $i$ is odd,}\\
        F_{i}-K_{i}^{-1}\ad(E_{i-1}E_{i+1})(E_{i})&
        \mbox{ if $i$ is even.}
      \end{cases}
\end{align*}
\begin{rema}
Let $T_{w_X}=T_1 T_3 \dots T_{n-1}$ denote the Lusztig automorphism corresponding to
the element $w_X=s_1s_3\dots s_{n-1}\in W$. The generators
\eqref{even} can be written as $F_{i}- K_{i}^{-1}T_{w_X}(E_{i})$. This shows for even $i$ that $B_i$ is a $q$-analog of   
$f_i+\theta(f_i)$.
\end{rema}
From \cite[Theorem 7.1]{a-Letzter03} we know how to write
$\uqpk$ in terms of generators and relations. Let $\cM^{\ge 0}$ denote
the subalgebra of $\uqpk$ generated by the elements of the set $\{E_i,K_i,
K_i^{-1}\,|\,\mbox{ $i$ odd}\}$. 
\begin{prop}
  The algebra $\uqpk$ is generated over $\cM^{\ge 0}$ by
  elements $B_i$, $1\le i \le n-1$, subject only to the following
  relations:
  \begin{enumerate}
    \item $K_i B_j K_i^{-1}=q^{a_{ij}} B_j$ for $1\le i,j\le n-1$
      with $i$ odd,
    \item $E_i B_j - B_j E_i= \delta_{ij}(K_i-K_i^{-1})/(q-q^{-1})$ for
      $1\le i,j\le n-1$ with $i$ odd,
    \item $B_i B_j - B_j B_i =0$ if $a_{ij}=0$,
    \item $B_i^2B_j - (q+q^{-1}) B_i B_j B_i + B_j B_i^2=0$ if
      $a_{ij}=-1$ and $i$ odd.
    \item  If $a_{ij}=-1$ and $i$ even then
       \begin{align*}
         B_i^2B_j - &(q+q^{-1}) B_i B_j B_i + B_j B_i^2\\
       &=-q^{-1}\big((q-q^{-1})^2 F_jE_jE_{j|i}+(q^{-1}K_j^{-1}+qK_j)E_{j|i}\big)
       \end{align*}
       where 
       \begin{align*}
          j|i=\begin{cases}
                i+1& \mbox{if $j=i-1$,}\\
                i-1& \mbox{if $j=i+1$.}
              \end{cases}
       \end{align*}
  \end{enumerate}
\end{prop}
\begin{rema}\label{rem:NoumiMolevLetzter}
In \cite{a-Noumi96} Noumi defined a subalgebra $\uqtwk$ of $\uqg=U_q(\slfrak_{2m})$ which depends on an explicit solution of the reflection equation. This subalgebra is generated by the elements of a matrix $K$ given in \cite[(2.19)]{a-Noumi96}, \cite[3.6]{a-NS95}. The properties of the matrices $L^+$ and $L^-$ occurring in the former reference imply that $\uqtwk$ is a right coideal subalgebra of $\uqg$ \cite[Section 2.4]{a-Noumi96}, \cite[Lemma 6.3]{a-Letzter99a}. As stated in \cite[Remark 6.7]{a-Letzter99a} one can show by direct computation that $\uqpk=\uqtwk$ for $m=2$. This implies, for general $m$, that all generators of $\uqpk$ are contained in $\uqtwk$ and hence $\uqpk\subseteq \uqtwk$. On the other hand it was proved in \cite{MSRI-Letzter} that $\uqpk$ is a maximal right coideal subalgebra specializing to $U(\kfrak)$ in a suitable limit $q\rightarrow 1$. This implies, at least for generic $q$, that $\uqtwk=\uqpk$. In particular, up to notational changes, the algebra $\uqpk$ as defined above coincides with the algebra $\uqpk$ considered in \cite{a-MolRag08}. 
\end{rema}
\subsection{Braid group action}\label{sec:sp-bract}
By Corollary \ref{cor:Bract} one expects an action of $Br(\afrak_{m-1})$ on $\uqpk$ by algebra automorphisms. Moreover, the generators $\os_i$, for $i=1,\dots,m-1$ of $Br(\afrak_{m-1})$ should act in highest degree as $T_{2i} T_{2i-1} T_{2i+1} T_{2i}$ or its inverse. Observe that
\begin{align}\label{eq:TonFj1}
  T_{2i}^{-1} T_{2i-1}^{-1} T_{2i+1}^{-1} T_{2i}^{-1} (F_j)=
     \begin{cases}
        [[F_{2i},F_{2i-1}]_q,F_{2i-2}]_q & \mbox{if $j=2i-2$,}\\
        F_{2i+1} & \mbox{if $j=2i-1$,}\\
        F_{2i-1} & \mbox{if $j=2i+1$,}\\
        [[F_{2i},F_{2i+1}]_q,F_{2i+2}]_q & \mbox{if $j=2i+2$}\\
     \end{cases}
\end{align}
and
\begin{align}
  T_{2i}^{-1} T_{2i-1}^{-1} &T_{2i+1}^{-1} T_{2i}^{-1} (-K_{2i}^{-1} T_{w_X}(E_{2i}))\nonumber\\ 
  &= - K_{2i-1} K_{2i} K_{2i+1}  T_{2i}^{-1} T_{2i-1}^{-1} T_{2i+1}^{-1} T_{2i}^{-1} T_{2i-1} T_{2i+1} (E_{2i})\nonumber\\
  &=-K_{2i-1} K_{2i} K_{2i+1}  T_{2i-1} T_{2i+1} T_{2i}^{-1} T_{2i-1}^{-1} T_{2i+1}^{-1} T_{2i}^{-1}(E_{2i})\nonumber\\
  &=K_{2i-1} K_{2i+1}  T_{2i-1}^2 T_{2i+1}^2 (F_{2i})\nonumber\\
  &=(q-q^{-1})^2[[F_{2i},F_{2i+1}]_q,F_{2i-1}]_q E_{2i-1}E_{2i+1} \label{eq:TonFj2}\\ 
  &\qquad -  q^2(q-q^{-1}) [F_{2i},F_{2i+1}]_q K_{2i-1}E_{2i+1}\nonumber\\ 
  &\qquad -  q^2(q-q^{-1})[F_{2i},F_{2i-1}]_q K_{2i+1} E_{2i-1} + q^4 F_{2i} K_{2i-1} K_{2i+1}.\nonumber 
\end{align}
Equation \eqref{eq:TonFj1} motivates the following definition. For $i=1,\dots,m-1$ and $j=1,3,5,\dots,2m-1$ define
\begin{align}
  \tau_i^-(X_j)=\begin{cases}
                  X_{2i+1} & \mbox{if $j=2i-1$},\\
                  X_{2i-1} & \mbox{if $j=2i+1$},\\
                  X_j      & \mbox{if $j$ is odd and $j\notin\{2i-1,2i+1\}$}
                \end{cases}\label{eq:tauimXdef}
\end{align}
where $X$ denotes any of the symbols $F, E, K$, or $K^{-1}$.
Moreover, Equations \eqref{eq:TonFj1} and \eqref{eq:TonFj2} suggest the following definition up to the insertion of additional $q$-factors. For $j\neq 2i$ even define
\begin{align}\label{eq:tauimBj1}
  \tau_i^-(B_j)=\begin{cases}
            q^{-1/2}[[B_{2i},B_{2i-1}]_q,B_{2i-2}]_q & \mbox{if $j=2i-2$,}\\
            q^{-1/2}[[B_{2i},B_{2i+1}]_q,B_{2i+2}]_q & \mbox{if $j=2i+2$,}\\
            B_j & \mbox{if $j\notin\{2i-2,2i,2i+2\}$,}\\
                \end{cases}
\end{align}
 and in the case $j=2i$ define
\begin{align}
  \tau_i^-(B_{2i}) &=q^{-1}(q-q^{-1})^2[[B_{2i},B_{2i+1}]_q,B_{2i-1}]_q E_{2i-1}E_{2i+1} \label{eq:tauimBj2}\\ 
  &\qquad -  q(q-q^{-1}) [B_{2i},B_{2i+1}]_q K_{2i-1}E_{2i+1}\nonumber\\ 
  &\qquad -  q(q-q^{-1})[B_{2i},B_{2i-1}]_q K_{2i+1} E_{2i-1} + q^3 B_{2i} K_{2i-1} K_{2i+1}.\nonumber 
\end{align}
\begin{thm}\label{thm:tauimIII}
  Let $1\le i \le m-1$.\\
  1) There exists a unique algebra automorphism $\tau_i^-$ of $\uqpk$ such that $\tau_i^-(B_j)$ is given by \eqref{eq:tauimBj1} and \eqref{eq:tauimBj2} for even $j$ with $2\le j\le 2m-2$, and $\tau_i^-(X_j)$ is given by \eqref{eq:tauimXdef} for odd $j$ with $1\le j \le 2m-1$. \\
  2) The inverse automorphism $\tau_i$ of $\tau_i^-$ is determined by
  $\tau_i(X_j)=\tau_i^-(X_j)$ for odd $j$ with $1\le j \le 2m-1$, where the symbol $X$ denotes any of the symbols $E, F, K,$ and $K^{-1}$, and by
\begin{align}\label{eq:tauiBj1}
    \tau_i(B_j)=\begin{cases}
            q^{-1/2}[[B_{2i-2},B_{2i-1}]_q,B_{2i}]_q & \mbox{if $j=2i-2$,}\\
            q^{-1/2}[[B_{2i+2},B_{2i+1}]_q,B_{2i}]_q & \mbox{if $j=2i+2$,}\\
            B_j & \mbox{if $j\notin\{2i-2,2i,2i+2\}$, $j$ even,}\\
                  \end{cases}
\end{align}    
\begin{align}
   \tau_i(B_{2i}) &=q^{-1}(q-q^{-1})^2[B_{2i+1},[B_{2i-1},B_{2i}]_q]_q E_{2i-1}E_{2i+1} \label{eq:tauiBj2}\\ 
  &\qquad -  q^{-2}(q-q^{-1}) [B_{2i+1},B_{2i}]_q K_{2i-1}E_{2i+1}\nonumber\\ 
  &\qquad -  q^{-2}(q-q^{-1})[B_{2i-1},B_{2i}]_q K_{2i+1} E_{2i-1} + q^{-3} B_{2i} K^{-1}_{2i-1} K^{-1}_{2i+1}.\nonumber 
\end{align}
3) The relation $\tau_{i-1} \tau_i \tau_{i-1}= \tau_i \tau_{i-1} \tau_i$ holds if $2\le i\le m-1$. Moreover, $\tau_i \tau_j =\tau_j \tau_i$ if $|i-j|\neq 1$ and $1\le j \le m-1$.
\end{thm}
It suffices to verify Theorem \ref{thm:tauimIII} in the case where $\gfrak=\slfrak_8(\C)$. In this case all necessary checks are performed by the \texttt{GAP}-code \texttt{III-A7.txt} which is available from \cite{WWW}.
\begin{rema}\label{rem:q-analog}
  If one sets $q$ and $K_j$ equal to $1$ in Formulas \eqref{eq:tauiBj1} and \eqref{eq:tauiBj2} then one obtains precisely Formula \eqref{eq:Adbj} from the classical  case. It is in this sense that the action of $\tau_i$ on $\uqpk$ is a deformation of the action of $\Ad(s_{2i} s_{2i-1} s_{2i+1} s_{2i})$ on $U(\mathfrak{sp}_{2m}(\C))$. Most interesting are the higher degree terms in \eqref{eq:tauimBj2} and \eqref{eq:tauiBj2} which we were only able to find by comparison with the Lusztig action given by \eqref{eq:TonFj2}.
\end{rema}
\subsection{Action of $\Z^m\rtimes Br(\afrak_{m-1})$ on $\uqpk$}
We now combine the action of $Br(\afrak_{m-1})$ on $\uqpk$ constructed in the previous subsection with the action of $\Z^m$ on $\uqpk$ given by the Lusztig automorphisms $T_i$ for $i=1,3,\dots,2m-1$. In view of Remarks \ref{rem:NoumiMolevLetzter} and \ref{rem:q-analog} this will provide a proof of Theorem \ref{conj:MR}.
\begin{lem}\label{lem:LusztigInvariance}
  Let $j\in\{1,\dots,2m-1\}$ be odd. Then $T_j(\uqpk)=\uqpk$.
\end{lem}
\begin{proof}
  It suffices to prove that $T_j^{-1}(\uqpk)=\uqpk$.
  The Lusztig automorphism $T_j^{-1}$ maps the generators \eqref{odd}
  to $\uqpk$. Hence it remains to show that $T_j^{-1}(B_i)\in \uqpk$ if  $|i-j|=1$. To this end assume that $j=i-1$ and calculate
  \begin{align*}
    F_{i-1}K_i^{-1}&\ad(E_{i-1}E_{i+1})(E_i)- q
    K_i^{-1}\ad(E_{i-1}E_{i+1})(E_i) F_{i-1}\\
    &= q K_i^{-1}\big(F_{i-1}\ad(E_{i-1}E_{i+1})(E_i) -\ad(E_{i-1}E_{i+1})(E_i)
    F_{i-1}\big)\\
    &= q K_i^{-1} \ad(F_{i-1}) \big(\ad(E_{i-1}E_{i+1})(E_i) \big)K_{i-1}^{-1} \\
    &= K_{i-1}^{-1}K_i^{-1} \ad(E_{i+1})E_i\\
    &= T_{i-1}^{-1}(K_i^{-1} T_{i-1}T_{i+1}(E_i)).
  \end{align*}
  Hence one obtains 
  \begin{align*}
    T_{i-1}^{-1}(B_i)&=T_{i-1}^{-1}(F_i) - T_{i-1}^{-1}(K_i^{-1}
    T_{i-1}T_{i+1}(E_i))\\
     &= F_{i-1} B_i - q B_i F_{i-1}           
  \end{align*}
  and this element does belong to $\uqpk$.
\end{proof}
The Lusztig automorphisms $T_j$ for $j=1,3,5,\dots,2m-1$ commute pairwise. Hence the above lemma produces an action of the additive group $\Z^m$ on $\uqpk$ by algebra automorphisms. The braid group $Br(\afrak_{m-1})$ acts on $\Z^m$ from the left and one may hence form the semidirect product $\Z^m\rtimes Br(\afrak_{m-1})$ as follows. Let $\mathrm{e}_i$, $i=1,\dots,m$, denote the standard basis of $\Z^m$. Then the multiplication on $\Z^m\rtimes Br(\afrak_{m-1})$ is determined by \begin{align}
  (\mathrm{e}_i,\sigma)\cdot (\mathrm{e}_{i'},\mu)=(\mathrm{e}_i + \mathrm{e}_{\pi(\sigma)(i')},\sigma\circ\mu)
\end{align}
where $\pi:Br(\afrak_{m-1})\rightarrow S_{m}$ denotes the canonical projection onto the symmetric group in $m$ elements. By the following theorem the actions of $\Z^m$ and $Br(\afrak_{m-1})$ on $\uqpk$ give the desired action of $\Z^m\rtimes Br(\afrak_{m-1})$.
\begin{thm}
  There exists a unique group homomorphism 
  \begin{align*}
     \Z^m\rtimes Br(\afrak_{m-1})\rightarrow \Aut_{alg}(\uqpk)
  \end{align*}
  such that $(\mathrm{e}_i,1)\mapsto T_{2i-1}$ for $i=1,\dots, m$, and $(0,\os_j)\mapsto \tau_j$ for $j=1,\dots,m-1$.
\end{thm} 
\begin{proof}
  To prove the theorem one needs to verify the following equalities
  \begin{align}
    \tau_i \circ T_{2i-1}&= T_{2i+1} \circ \tau_i, & \tau_i \circ T_{2i+1}&=   T_{2i-1} \circ \tau_i &&\mbox{for $i=1,\dots,m{-}1$,} \label{eq:smash1}\\
  &&  \tau_j \circ T_{2i-1}&= T_{2i-1} \circ\tau_j &&\mbox{if $j\neq   i,i-1$}\label{eq:smash2}
  \end{align}
  of automorphisms of $\uqpk$. Again, it suffices to verify these equations on the generators of $\uqpk$ as both sides of the equations are already known to be algebra automorphisms. It is straightforward to check Equations \eqref{eq:smash1} and \eqref{eq:smash2} when evaluated on the elements $\{E_j, F_j, K_j^{\pm 1}\,|\,j \mbox{ odd}\}$. To obtain Equation \eqref{eq:smash2} it hence suffices to check that 
\begin{align*}
   \tau_i \circ T_{2i-3}^{-1}(B_{2i-2})&=  T_{2i-3}^{-1} \circ \tau_i (B_{2i-2}),& 
   \tau_j \circ T_{2j+3}^{-1}(B_{2j+2})&=  T_{2j+3}^{-1} \circ \tau_j (B_{2j+2})
\end{align*}
for $i=2,\dots,m-1$ and $j=1,\dots,m-2$. This follows from \eqref{eq:tauiBj1} and from the relations $T_{2i-3}^{-1}(B_{2i-2})=[B_{2i-3},B_{2i-2}]_q$ and
$T_{2i+3}^{-1}(B_{2i+2})=[B_{2i+3},B_{2i+2}]_q$ which were verified in the proof of Lemma \ref{lem:LusztigInvariance}.

For symmetry reasons it now suffices to verify the first equation of \eqref{eq:smash1}. To complete the proof of the theorem we hence have to show that
\begin{align}
  \tau_i \circ T_{2i-1}^{-1}(B_{2i-2})&= T_{2i+1}^{-1}\circ\tau_i(B_{2i-2}),\label{eq:smash3}\\
  \tau_i \circ T_{2i-1}^{-1}(B_{2i})&= T_{2i+1}^{-1}\circ\tau_i(B_{2i}), \label{eq:smash4}\\
  \tau_i \circ T_{2i-1}^{-1}(B_{2i+2})&= T_{2i+1}^{-1}\circ\tau_i(B_{2i+2}). \label{eq:smash5}
\end{align}
Equation \eqref{eq:smash3} follows from $T_{2i-1}^{-1}(B_{2i-2})=[B_{2i-1},B_{2i-2}]_q$ and $T_{2i+1}^{-1}(B_{2i})=[B_{2i+1},B_{2i}]_q$ and the definition of $\tau_i$ in Theorem \ref{thm:tauimIII}.2). Equations \eqref{eq:smash4} and \eqref{eq:smash5} are equivalent to
\begin{align}\label{eq:smash6}
  [\tau_i(B_{2i-1}),\tau_i(B_{2i})]_q &= T_{2i+1}^{-1}(\tau_{i}(B_{2i})),\\
  \tau_i(B_{2i+2})&=  T_{2i+1}^{-1}(\tau_{i}(B_{2i+2})),\label{eq:smash7}
\end{align}
respectively, which are checked by computer calculations at the end of the file \texttt{III-A7.txt}. This concludes the proof of the Theorem.
\end{proof}
\providecommand{\bysame}{\leavevmode\hbox to3em{\hrulefill}\thinspace}
\providecommand{\MR}{\relax\ifhmode\unskip\space\fi MR }
\providecommand{\MRhref}[2]{%
  \href{http://www.ams.org/mathscinet-getitem?mr=#1}{#2}
}
\providecommand{\href}[2]{#2}

\end{document}